\documentclass[a4paper,12pt,onecolumn]{article}
\usepackage[vmargin=2cm,hmargin=2cm,headheight=14.5pt,top=2cm,headsep=.5cm]{geometry}
\usepackage{mathptmx}
\usepackage{bm}
\usepackage{empheq}
\usepackage{stackrel}
\usepackage{cases}
\usepackage{amsthm,amsmath,amscd}
\usepackage{amssymb,amsfonts}
\usepackage{makeidx}
\usepackage[all]{xy}
\usepackage{mathptmx}
\usepackage{perpage}
\usepackage[symbol]{footmisc}

\MakePerPage[2]{footnote}
\usepackage{graphicx}
\DeclareMathSizes{12}{12}{8}{6}

\usepackage{cite}
\newtheoremstyle{ptheorem}{1em}{0em}{\itshape}{}{\bfseries}{.}{.5em}{}
\theoremstyle{ptheorem}
\newtheorem{thm}{Theorem}[section]
\newtheorem{pro}[thm]{Proposition}
\newtheorem{lem}[thm]{Lemma}
\newtheorem{cor}[thm]{Corollary}

\theoremstyle{definition}
\newtheorem{dfn}{Definition}[section]

\theoremstyle{remark}
\newtheorem{exa}{Example}[section]

\newtheorem{rem}{Remark}[section]

\numberwithin{equation}{section}
\numberwithin{figure}{section}

\DeclareMathOperator{\dif}{d}

\DeclareMathOperator{\arctanh}{arctanh}
\newcommand{\cH}{{\mathcal H}}

\newcommand{\bR}{{\mathbb R}}

\newcommand{\bZ}{{\mathbb Z}}

\renewcommand{\a}{\alpha}
\renewcommand{\b}{\beta}
\renewcommand{\c}{\gamma}
\renewcommand{\l}{\lambda}
\newcommand{\e}{\epsilon}

\renewcommand{\phi}{\varphi}

\newcommand{\fa}{\forall}

\newcommand{\nkp}{\enskip}

\newcommand{\sfa}{\nkp\fa}

\renewcommand{\(}{\left(}
\renewcommand{\)}{\right)}
\renewcommand{\[}{\left[}
\renewcommand{\]}{\right]}

\newcommand{\til}{\tilde}

\begin{document}
\title{Solutions and Green's function of the first order linear equation with reflection and initial conditions\footnote{Partially supported by FEDER and Ministerio de Educaci\'on y Ciencia, Spain, project MTM2010-15314}}

\author{
Alberto Cabada \, and F. Adri\'an F. Tojo\footnote{Supported by  FPU scholarship, Ministerio de Educaci\'on, Cultura y Deporte, Spain.} \\
\normalsize
Departamento de An\'alise Ma\-te\-m\'a\-ti\-ca, Facultade de Matem\'aticas,\\ 
\normalsize Universidade de Santiago de Com\-pos\-te\-la, Spain.\\ 
\normalsize e-mail: alberto.cabada@usc.es, fernandoadrian.fernandez@usc.es}
\date{}

\maketitle

\begin{abstract}
This work is devoted to the study of the existence and sign of Green's functions for first order linear problems with constant coefficients and initial (one point) conditions. We first prove a result on the existence of solutions of $n$-th order linear equations with involutions via some auxiliary functions to later prove a uniqueness result in the first order case. We study then different situations for which a Green's function can be obtained explicitly and derive several results in order to obtain information about the sign of the Green's function. Once the sign is known, optimal maximum and anti-maximum principles follow.
\end{abstract}

\section{Introduction}

The study of functional differential equations with involutions (DEI) can be traced back to the solution of the equation $x'(t)=x(1/t)$ by Silberstein (see \cite{Sil}) in 1940. Briefly speaking, an involution is just a function $f$ that satisfies $f(f(x))=x$ for every $x$ in its domain of definition. For most applications in analysis, the involution is defined on an interval of $\bR$ and in the majority of the cases, it is continuous, which implies it is decreasing and has a unique fixed point. Ever since that foundational paper of Siberstein, the study of problems with DEI has been mainly focused on those cases with initial conditions, with an extensive research in the case of the reflection $f(x)=-x$.\par
Wiener and Watkins study in \cite{Wie} the solution of the equation $x'(t)-a\, x(-t)=0$ with initial conditions. Equation $x'(t)+a\, x(t)+b\,x(-t)=g(t)$ has been treated by Piao in \cite{Pia, Pia2}. In \cite{Kul, Sha, Wie, Wat1, Wie2} some results are introduced to transform this kind of problems with involutions and initial conditions into second order ordinary differential equations with initial conditions or first order two dimensional systems, granting that the solution of the last will be a solution to the first. Furthermore, asymptotic properties and boundedness of the solutions of initial first order problems are studied in \cite{Wat2} and \cite{Aft} respectively. Second order boundary value problems have been considered in \cite{Gup, Gup2, Ore2, Wie2} for Dirichlet and Sturm-Liouville boundary value conditions, higher order equations has been studied in \cite{Ore}. Other techniques applied to problems with reflection of the argument can be found in \cite{And, Ma, Wie1}.\par
More recently, the papers of Cabada et al. \cite{Cab4, Cab5} have further studied the case of the second order equation with two-point boundary conditions, adding a new element to the previous studies: the existence of a Green's function. Once the study of the sign of the aforementioned function is done, maximum and anti-maximum principles follow. Other works in which Green's functions are obtained  for functional differential equations (but with a fairly different setting, like delay or normal equations) are, for instance, \cite{AzDo1, AzDo2, Aga, Dom1, Dom2,Dom3}.\par
In this paper we try to answer to the following question: How is it possible find a solution of an initial problem with a differential equation with reflection? What is more, in which cases can a Green's function be constructed and how can it be found?\par
Section 2 will have two parts. In the first one we construct the solutions of the $n$-th order DEI with reflection, constant coefficients and initial conditions. In the second one we find the Green's function for the order one case. In Section 3 we apply these findings in order to describe exhaustively the range of values for which suitable comparison results are fulfilled and we illustrate them with some examples.

\section{Solutions of the initial problem}
In order to prove an existence result for the  $n$-th order DEI with reflection, we consider the even and odd parts of a function $f$, that is $f_e(x):=[f(x)+f(-x)]/2$ and $f_o(x):=[f(x)-f(-x)]/2$ as done in \cite{Cab4}.
\subsection{The $n$-th order problem}
Consider the following  $n$-th order DEI with involution
\begin{equation}\label{hordp} Lu:=\sum_{k=0}^n\[a_ku^{(k)}(-t)+b_ku^{(k)}(t)\]=h(t),\ t\in\bR;\quad u(t_0)=c,\end{equation}
where  $h\in L_{\operatorname{loc}}^1(\bR)$, $t_0$, $c$, $a_k$, $b_k\in\bR$ for $k=0,\dots n-1$; $a_n=0$; $b_n=1$. A solution to this problem will be a function $u\in W^{n,1}_{\operatorname{loc}}(\bR)$, that is, $u$ is $k$ times differentiable in the sense of distributions and each of the derivatives satisfies $u^{k)}|_K\in L^1(K)$ for every compact set $K\subset\bR$.
%, and one of the two following conditions holds:
%\begin{enumerate}
%\item[$(i)$] $a_i=0$ for all $i \in \{0,Ê\ldots,n-1\}$ such that $n+i$ is even.
%\item[$(ii)$] $a_i=0$ for all $i \in \{0,Ê\ldots,n-1\}$ such that $n+i$ is odd.
%\end{enumerate}

\begin{thm}\label{thmconstsoln} 
Assume that there exist $\til u$ and $\til v$, functions such that satisfy
\begin{align}\sum_{i=0}^{n-j}\binom{i+j}{j}\[(-1)^{n+i-1}a_{i+j}\til u^{(i)}(-t)+b_{i+j}\til u^{(i)}(t)\]= & 0,\  t\in\bR;\ j=0,\dots,n-1,\\
\sum_{i=0}^{n-j}\binom{i+j}{j}\[(-1)^{n+i}a_{i+j}\til v^{(i)}(-t)+b_{i+j}\til v^{(i)}(t)\]= & 0,\  t\in\bR;\ j=0,\dots,n-1, \\ (\til u_e\til v_e-\til u_o\til v_o)(t)\ne & 0,\ t\in\bR.\end{align}
and also one of the following
\begin{align*}
(h1) & \quad L\,\til u=0\ \text{and}\ \til u (t_0)\ne 0,\\
(h2) & \quad L\,\til v=0\ \text{and}\ \til v (t_0)\ne 0,\\
(h3) & \quad a_0+b_0\ne 0\ \text{and}\  (a_0+b_0)\int_0^{t_0}(t_0-s)^{n-1}\frac{\til v(t_0)\til u_e(s)-\til u(t_0)\til v_o(s)}{(\til u_e\til v_e-\til u_o\til v_o)(s)}\dif s\ne 1.
\end{align*}
Then problem \eqref{hordp} has a solution.
\end{thm}
\begin{proof}
Define
$$\phi:=\frac{h_o\til v_e-h_e\til v_o}{\til u_e\til v_e-\til u_o\til v_o},\quad\text{and}\quad \psi:=\frac{h_e\til u_e-h_o \til u_0}{\til u_e\til v_e-\til u_o\til v_o}.$$
Observe that $\phi$ is odd, $\psi$ is even and $h=\phi \til u+\psi\til v$. So, in order to ensure the existence of solution of problem \eqref{hordp} it is enough to find $y$ and $z$ such that $Ly=\phi\til u$ and $Lz=\psi \til v$ for, in that case, defining $u=y+z$, we can conclude that $Lu=h$. We will deal with the initial condition later on.\par
Take $y=\til\phi\,\til u$, where $$\til\phi(t):=\int_{0}^t\int_{0}^{s_{n}}\cdots\int_{0}^{s_2}\phi(s_1)\dif s_1\stackrel{n}{\cdots}\dif s_n=\frac{1}{(n-1)!}\int_0^t(t-s)^{n-1}\phi(s)\dif s.$$ Observe that $\til\phi$ is even if $n$ is odd and vice-versa. In particular, we have that 
$$\til \phi^{(j)}(t)= (-1)^{j+n-1} \; \til \phi^{(j)}(-t)\,, \quad j=0, \ldots, n.$$
%and
%$$\til \psi^{(j)}(t)=(-1)^{j+n} \, \til \psi^{(j)}(-t)\, , \quad j=0, \ldots, n.$$
% (\til \phi no se define hasta después)
Thus,
\begin{align*}
Ly(t) & =\sum_{k=0}^n\[a_k(\til\phi\til u)^{(k)}(-t)+b_k(\til\phi\til u)^{(k)}(t)\]\\
&=\sum_{k=0}^n\sum_{j=0}^k\binom{k}{j}\[(-1)^ka_k\til\phi^{(j)}(-t)\til u^{(k-j)}(-t)+b_k\til\phi^{(j)}(t)\til u^{(k-j)}(t)\] \\ & =\sum_{k=0}^n\sum_{j=0}^k\binom{k}{j}\til\phi^{(j)}(t)\[(-1)^{k+j+n-1}a_k\til u^{(k-j)}(-t)+b_k\til u^{(k-j)}(t)\] \\ &  =\sum_{j=0}^n\til\phi^{(j)}(t)\sum_{k=j}^n\binom{k}{j}\[(-1)^{k+j+n-1}a_k\til u^{(k-j)}(-t)+b_k\til u^{(k-j)}(t)\] \\ &
=\sum_{j=0}^n\til\phi^{(j)}(t)\sum_{i=0}^{n-j}\binom{i+j}{j}\[(-1)^{i+n-1}a_{i+j}\til u^{(i)}(-t)+b_{i+j}\til u^{(i)}(t)\]=\til\phi^{(n)}(t)\til u(t)=\phi(t)\til u(t).
\end{align*}
Hence, $Ly=\phi \til u$. 

All the same, by taking $z=\til\psi\til v$ with $\til\psi(t):=\frac{1}{(n-1)!}\int_0^t(t-s)^{n-1}\psi(s)\dif s$, we have that $Lz=\psi \til v$.
\par
Hence, defining $\bar{u}:=y+z=\til\phi\,\til u+\til\psi\til v$ we have that $\bar{u}$ satisfies $L\,\bar{u}=h$ and $\bar{u}(0)=0$. 

If we assume $(h1)$, $w=\bar u+\frac{c-\bar{u}(t_0)}{\til u(t_0)}\til u$ is clearly a solution of problem \eqref{hordp}.

When $(h2)$ is fulfilled a solution of problem \eqref{hordp} is given by $w=\bar u+\frac{c-\bar{u}(t_0)}{\til v(t_0)}\til v$.

 If $(h3)$ holds, using the aforementioned construction we can find $w_1$ such that $L\,w_1=1$ and $w_1(0)=0$. Now, $w_2:=w_1-1/(a_0+b_0)$ satisfies $L\,w_2=0$. Observe that the second part of condition $(h2)$ is precisely $w_2(t_0)\ne 0$, and hence, defining $w=\bar u+\frac{c-\bar{u}(t_0)}{w_2(t_0)} w_2$ we have that $w$ is a solution of problem \eqref{hordp}.
\end{proof}
\begin{rem}
Having in mind condition $(h1)$ in Theorem \ref{thmconstsoln}, it is immediate to verify that $L\, \til u=0$ provided that 
\begin{center}
$a_i=0$ for all $i \in \{0,Ê\ldots,n-1\}$ such that $n+i$ is even.
\end{center}

In an analogous way for $(h2)$, one can show that  $L\, \til v=0$ when 
\begin{center}
 $a_i=0$ for all $i \in \{0,Ê\ldots,n-1\}$ such that $n+i$ is odd.
\end{center}
\end{rem}
\subsection{The first order problem}
After proving the general result for the $n$-th order case, we concentrate our work in the first order problem
\begin{equation}\label{gpabconst} u'(t)+a\,u(-t)+b\,u(t)=h(t),\text{ for a.\,e. } t\in \bR;\quad u(t_0)=c,\end{equation}
with $h\in L^1_{\operatorname{loc}}(\bR)$ and $t_0$, $a$, $b$, $c\in\bR$. A solution of this problem will be $u\in W^{1,1}_{\operatorname{loc}}(\bR)$.

In order to do so, we first study the homogeneous equation
\begin{equation}\label{heabconst} u'(t)+a\,u(-t)+b\,u(t)=0,\ t\in \bR.\end{equation}
By differentiating and making the proper substitutions we arrive to the equation
\begin{equation}\label{rheabconst} u''(t)+(a^2-b^2)u(t)=0,\ t\in \bR.\end{equation}
Let $\omega:=\sqrt{|a^2-b^2|}$. Equation \eqref{rheabconst} presents three different cases:\par
\textbf{(C1). $a^2>b^2$.} In such a case, $u(t)=\a\cos \omega t+\b\sin\omega t$ is a solution of  \eqref{rheabconst} for every $\a,\,\b\in\bR$. If we impose equation \eqref{heabconst} to this expression we arrive to the general solution
$$u(t)=\a(\cos\omega t-\frac{a+b}{\omega}\sin \omega t)$$
of equation \eqref{heabconst} with $\a\in\bR$.\par 
\textbf{(C2). $a^2<b^2$.} Now, $u(t)=\a\cosh \omega t+\b\sinh\omega t$ is a solution of  \eqref{rheabconst} for every $\a,\,\b\in\bR$. To get equation \eqref{heabconst}  we arrive to the general solution
$$u(t)=\a(\cosh\omega t-\frac{a+b}{\omega}\sinh \omega t)$$
of equation \eqref{heabconst} with $\a\in\bR$.\par 
\textbf{(C3). $a^2=b^2$.} In this a case, $u(t)=\a t+\b$ is a solution of  \eqref{rheabconst} for every $\a,\,\b\in\bR$. So, equation \eqref{heabconst} holds provided that one of the two following cases is fulfilled:

\textbf{(C3.1). $a=b$,} where
$$u(t)=\a(1-2\,a\ t)$$
is the general solution of equation \eqref{heabconst} with $\a\in\bR$, and\par 
\textbf{(C3.2). $a=-b$,} where
$$u(t)=\a$$
is the general solution of equation \eqref{heabconst} with $\a\in\bR$.\par 

Now, according to Theorem \ref{thmconstsoln}, we denote $\til u$, $\til v$ satisfying
\begin{align}\label{v1}\til u'(t)+a\til u(-t)+b\til u(t) & =0, \quad \til u(0)=1,\\ \label{v2}
\til v'(t)-a\til v(-t)+b\til v(t) & =0,\quad \til v(0)=1.\end{align}
Observe that $\til u$ and $\til v$ can be obtained from the explicit expressions of the cases (C1)--(C3) by taking $\a=1$.
\begin{rem} Note that if $u$ is in the case (C3.1), $v$ is in the case (C3.2) and vice-versa.\end{rem}
We have now the following properties of functions $\til u$ and $\til v$.
\begin{lem} \label{auxlemuv}For every $t,s\in\bR$, the following properties hold.
\begin{enumerate}
\item $\til u_e\equiv\til v_e$, $\til u_o\equiv k\,\til v_o$ for some real constant $k$ a.e.,
\item $\til u_e(s)\til v_e(t)=\til u_e(t)\til v_e(s)$, $\til u_o(s)\til v_o(t)=\til u_o(t)\til v_o(s)$,
\item $\til u_e\til v_e-\til u_o\til v_o\equiv 1$.
\item $\til u(s)\til v(-s)+\til u(-s)\til v(s)=2[\til u_e(s)\til v_e(s)-\til u_o(s)\til v_o(s)]=2.$
\end{enumerate}
\end{lem}
\begin{proof}
$(I)$ and $(III)$ can be checked by inspection of the different cases. $(II)$ is a direct consequence of $(I)$. $(IV)$ is obtained from the definition of even and odd parts and $(III)$.
\end{proof}
Now, Theorem \ref{thmconstsoln} has the following corollary.
\begin{cor}\label{corconstsoln} Problem  \eqref{gpabconst} has a unique solution if and only if  $\til u(t_0)\ne 0$.
\end{cor}
\begin{proof} Considering Lemma \ref{auxlemuv} ($III$), $\til u$ and $\til v$, defined as in \eqref{v1} and \eqref{v2} respectively, satisfy the hypothesis of  Theorem  \ref{thmconstsoln}, $(h1)$, therefore a solution exists.\par
Now, assume $w_1$ and $w_2$ are two solutions of \eqref{gpabconst}. Then $w_2-w_1$ is a solution of \eqref{heabconst}. Hence, $w_2-w_1$ is of one of the forms covered in the cases (C1)--(C3) and, in any case, a multiple of $\til u$, that is $w_2-w_1=\l\,\til u$ for some $\l\in\bR$. Also, it is clear that $(w_2-w_1)(t_0)=0$, but we have $\til u(t_0)\neq 0$ as a hypothesis, therefore $\l=0$ and $w_1=w_2$. This is, problem  \eqref{gpabconst} has a unique solution.\par
Assume now that $w$ is a solution of \eqref{gpabconst} and $\til u(t_0)=0$. Then $w+\l\,\til u$ is also a solution of \eqref{gpabconst} for every $\l\in\bR$, which proves the result.
\end{proof}

%\begin{rem} \label{fcrem} Observe that $\til u_e=\til v_e$ in every case. Also, $\til u_o$ and $\til v_o$ are related by a constant, i.e. $\til u_o=k\,\til v_o$ a.e. This implies that $\til u_o(t)\til v_o(s)=\til u_o(s)\til v_o(t)$ for all $t\in\bR$.\end{rem}
This last Theorem raises an obvious question: In which circumstances $\til u(t_0)\neq0$? In order to answer this question, it is enough to study the cases (C1)--(C3). We summarize this study in the following Lemma which can be checked easily.
\begin{lem}
 $\til u(t_0)=0$ only in the following cases,
 \begin{itemize}
\item if $a^2>b^2$ and $t_0=\frac{1}{\omega}\arctan\frac{\omega}{a+b}+k\pi$ for some $k\in\bZ$,
\item if $a^2<b^2$, $a\,b>0$ and $t_0=\frac{1}{\omega}\arctanh \frac{\omega}{a+b}$,
\item if $a=b$ and $t_0=\frac{1}{2a}$.
 \end{itemize}.
\end{lem}

\begin{dfn} Let $t_1,t_2\in\bR$. We define the\textbf{ oriented characteristic function} of the pair $(t_1,t_2)$ as
$$\chi_{t_1}^{t_2}(t):=\begin{cases}1, & t_1\le t\le t_2\\-1, & t_2\le t< t_1\\0, & \text{otherwise.}\end{cases}$$
\end{dfn} 
\begin{rem} The previous definition implies that, for any given integrable function $f:\bR\to\bR$,
$$\int_{t_1}^{t_2}f(s)\dif s=\int_{-\infty}^{\infty}\chi_{t_1}^{t_2}(s)f(s)\dif s.$$
Also, $\chi_{t_1}^{t_2}=-\chi_{t_2}^{t_1}$.
\end{rem}
The following corollary gives us the expression of the Green's function for problem \eqref{gpabconst}.
\begin{cor} Suppose $\til u(t_0)\ne0$. Then the unique solution of problem \eqref{gpabconst} is given by
$$u(t):=\int_{-\infty}^\infty G(t,s)h(s)\dif s+\frac{c-\bar{u}(t_0)}{\til u(t_0)}\til u(t),\quad t\in\bR,$$
where
\begin{equation}\label{eqsolgen}G(t,s):=\frac{1}{2}\([\til u(-s)\til v(t)+\til v(-s)\til u(t)]\chi_{0}^t(s)+[\til u(-s)\til v(t)-\til v(-s)\til u(t)]\chi_{-t}^{0}(s)\),\quad t,s\in\bR.\end{equation}
\end{cor}
\begin{proof} First observe that $G(t,\cdot)$ is bounded and of compact support for every fixed $t\in\bR$, so the integral $\int_{-\infty}^\infty G(t,s)h(s)\dif s$ is well defined. 
It is not difficult to verify, for any $t \in \bR$, the following equalities:
\begin{equation}
\label{equprima}
\begin{aligned}  
u'(t)-\frac{c-\bar{u}(t_0)}{\til u(t_0)}\til u'(t)  
= &\frac{1}{2}\left(\frac{\dif}{\dif t}\int_{0}^t\[\til u(-s)\til v(t)+\til v(-s)\til u(t)\]h(s)\dif s\right.\\
&\left. \;+\frac{\dif}{\dif t}\int_{-t}^{0}\[\til u(-s)\til v(t)-\til v(-s)\til u(t)\]h(s)\dif s\right)\\  
= &  \frac{1}{2}\left(\frac{\dif}{\dif t}\int_{0}^t \[\til u(-s)\til v(t)+\til v(-s)\til u(t)\]h(s)\dif s\right. \\
&\; \left.+\frac{\dif}{\dif t}\int_{0}^t\[\til u(s)\til v(t)-\til v(s)\til u(t)\]h(-s)\dif s\right)\\ = 
& \,h(t)+\frac{1}{2}\left(\int_{0}^t \[\til u(-s)\til v'(t)+\til v(-s)\til u'(t)\]h(s)\dif s\right.\\
& \left.  \; +\int_{0}^t\[\til u(s)\til v'(t)-\til v(s)\til u'(t)\]h(-s)\dif s\right).
\end{aligned}
\end{equation}
On the other hand,
\begin{equation}\label{eqrest}\begin{aligned} 
 a\bigg[u(-t)-&\frac{c-\bar{u}(t_0)}{\til u(t_0)}\til u(-t)\bigg]+b\bigg[u(t)-\frac{c-\bar{u}(t_0)}{\til u(t_0)}\til u(t)\bigg] \\ = & \frac{1}{2}a\int_{0}^{-t}\([\til u(-s)\til v(-t)+\til v(-s)\til u(-t)]h(s)+[\til u(s)\til v(-t)-\til v(s)\til u(-t)]h(-s)\)\dif s \\ & +  \frac{1}{2}b\int_{0}^{t}\([\til u(-s)\til v(t)+\til v(-s)\til u(t)]h(s)+[\til u(s)\til v(t)-\til v(s)\til u(t)]h(-s)\)\dif s \\ = & -\frac{1}{2}a\int_{0}^{t}\([\til u(s)\til v(-t)+\til v(s)\til u(-t)]h(-s)+[\til u(-s)\til v(-t)-\til v(-s)\til u(-t)]h(s)\)\dif s \\ & + \frac{1}{2} b\int_{0}^{t}\([\til u(-s)\til v(t)+\til v(-s)\til u(t)]h(s)+[\til u(s)\til v(t)-\til v(s)\til u(t)]h(-s)\)\dif s \\ = & \frac{1}{2} \int_{0}^{t}(-a[\til u(-s)\til v(-t)-\til v(-s)\til u(-t)]+b[\til u(-s)\til v(t)+\til v(-s)\til u(t)])h(s)\dif s\\  &+ \frac{1}{2} \int_{0}^{t}(-a[\til u(s)\til v(-t)+\til v(s)\til u(-t)]+b[\til u(s)\til v(t)-\til v(s)\til u(t)])h(-s)\dif s\\ = &
\frac{1}{2} \int_{0}^{t}(\til u(-s)[-a\til v(-t)+b\til v(t)]+\til v(-s)[a\til u(-t)+b\til u(t)]h(s)\dif s\\ &+  \frac{1}{2} \int_{0}^{t}(\til u(s)[-a\til v(-t)+b\til v(t)]-\til v(s)[a\til u(-t)+b\til u(t)])h(-s)\dif s\\ = &
-\frac{1}{2}\( \int_{0}^{t}(\til u(-s) \til v'(t)+\til v(-s)  \til u'(t))h(s)\dif s+ \int_{0}^{t}(\til u(s) \til v'(t)-\til v(s) \til u'(t))h(-s)\dif s\).
\end{aligned} \end{equation}
 Thus, adding \eqref{equprima} and \eqref{eqrest}, it is clear that $u'(t)+a\,u(-t)+b\,u(t)=h(t)$.\par
 We now check the initial condition.
 $$u(t_0)=c-\bar{u}(t_0)+\frac{1}{2}\int_{0}^{t_0}\([\til u(-s)\til v(t_0)+\til v(-s)\til u(t_0)]h(s)+[\til u(s)\til v(t_0)-\til v(s)\til u(t_0)]h(-s)\)\dif s.$$
 Using the construction of the solution provided in Theorem \ref{thmconstsoln}, it is an easy exercise to check that
 $$\bar{u}(t)=\frac{1}{2}\int_{0}^{t}\([\til u(-s)\til v(t)+\til v(-s)\til u(t)]h(s)+[\til u(s)\til v(t)-\til v(s)\til u(t)]h(-s)\)\dif s\sfa t\in\bR,$$
 which proves the result.
\end{proof}
Denote now $G_{a,b}$ the Green's function for problem \eqref{gpabconst} with coefficients $a$ and $b$. The following Lemma is analogous to \cite[Lemma 4.1]{Cab4}.
\begin{lem}\label{Gop} $ G_{a,b}(t,s)=-G_{-a,-b}(-t,-s), \quad \mbox{for all }  t,s\in I$.
\end{lem}
\begin{proof}
Let $u(t):=\int_{-\infty}^\infty G_{a,b}(t,s)h(s)\dif s$ be a solution to $u'(t)+a\,u(-t)+b\,u(t)=h(t)$. Let $v(t):=-u(-t)$. Then $v'(t)-a\,v(-t)-b\,v(t)=h(-t)$, and therefore $v(t)=\int_{-\infty}^\infty G_{-a,-b}(t,s)h(-s)\dif s$. On the other hand, by definition of $v$,
$$v(t)=-\int_{-\infty}^\infty G_{a,b}(-t,s)h(s)\dif s=-\int_{-\infty}^\infty G_{a,b}(-t,-s)h(-s)\dif s,$$
therefore we can conclude that $ G_{a,b}(t,s)=- G_{-a,-b}(-t,-s)$ for all $t,\;s\in I$.
\end{proof}

As a consequence of the previous result, we arrive at the following immediate conclusion.

\begin{cor} $G_{a,b}$ is positive if and only if $ G_{-a,-b}$ is negative on $I^2$.
\end{cor}
\section{Sign of the Green's Function}

In this section we use the above obtained expressions to obtain the explicit expression of the Green's function, depending on the values of the constants $a$ and $b$. Moreover we study the sign of the function and deduce suitable comparison results.

We separate the study in three cases, taking into consideration the expression of the general solution of equation \eqref{heabconst}.

\subsection{The case (C1)}
Now, assume the case $(C1)$, i.e., $a^2 > b^2$. Using equation \eqref{eqsolgen}, we get the following expression of $G$ for this situation:

$$G(t,s)=\[\cos(\omega (s - t)) + \frac{b}{\omega}\sin(\omega(s-t))\]\chi_0^t(s)+\frac{a}{\omega}\sin(\omega(s+t))\chi_{-t}^0(s),$$
which we can rewrite as
\begin{subequations}
\begin{empheq}[left={G(t,s)=\empheqlbrace}]{align}
   &  \cos\omega (s - t) + \frac{b}{\omega}\sin\omega(s-t), &  0\le s \le t, \label{G1} \\
     & \label{G2} -\cos\omega (s - t) - \frac{b}{\omega}\sin\omega(s-t), & t\le s \le 0, \\ 
     & \frac{a}{\omega}\sin\omega(s+t), &  -t\le s \le 0, \label{G3} \\
     & \label{G4} -\frac{a}{\omega}\sin\omega(s+t), &  0\le s \le -t, \\
     &      0, & \text{otherwise.}
 \end{empheq}
\end{subequations}\par
Studying the expression of $G$ we can obtain maximum and antimaximum principles. In order to do this, we will be interested in those maximal strips (in the sense of inclusion) of the kind $[\a,\b]\times\bR$ where $G$ does not change sign depending on the parameters.

So, we are in a position to study the sign of the Green's function in the different triangles of definition. The result is the following:

\begin{lem}\label{lempma1} Assume $a^2>b^2$ and define
$$\eta(a,b):=
\left\{
\begin{array}{lll}
\frac{1}{\sqrt{a^2-b^2}}\arctan \frac{\sqrt{a^2-b^2}}{b}, & \mbox{if}  & b>0,\\
\frac{\pi}{2|a|}, & \mbox{if}  & b=0,\\
\frac{1}{\sqrt{a^2-b^2}}\left(\arctan \frac{\sqrt{a^2-b^2}}{b}+\pi\right), & \mbox{if} & b<0.
\end{array}
\right.$$

 Then, the Green's function of problem \eqref{gpabconst} is
 \begin{itemize}
 \item  positive on $\{(t,s), \; 0<s<t\}$ if and only if $t \in (0,\eta(a,b))$,
 \item  negative on $\{(t,s), \; t<s<0\}$ if and only if $t \in (-\eta(a,-b),0)$.
 \end{itemize}
 If $a>0$, the Green's function of problem \eqref{gpabconst} is
\begin{itemize}
\item  positive on $\{(t,s), \; -t<s<0\}$ if and only if $t \in (0,\pi/ \sqrt{a^2-b^2})$,
\item  positive on $\{(t,s), \; 0<s<-t\}$ if and only if $t \in (-\pi/ \sqrt{a^2-b^2},0)$,
\end{itemize}
and, if $a<0$, the Green's function of problem \eqref{gpabconst} is
\begin{itemize}
\item  negative on $\{(t,s), \; -t<s<0\}$ if and only if $t \in (0,\pi/ \sqrt{a^2-b^2})$,
\item  negative on $\{(t,s), \; 0<s<-t\}$ if and only if $t \in (-\pi/ \sqrt{a^2-b^2},0)$.
\end{itemize}
\end{lem}
\begin{proof}For $0<b<a$, the argument of the $\sin$ in \eqref{G3} is positive, so \eqref{G3} is positive for $t<\pi/\omega$. On the other hand, it is easy to check that \eqref{G1} is positive as long as $t<\eta(a,b)$.\par
The rest of the proof continues similarly.
\end{proof}

As a corollary of the previous result we obtain the following one:

\begin{lem}\label{lempma1} Assume $a^2>b^2$. Then,
\begin{itemize}
\item if $a>0$, the Green's function of problem \eqref{gpabconst} is non-negative on $[0,\eta(a,b)]\times\bR$,
\item if $a<0$, the Green's function of problem \eqref{gpabconst} is non-positive on $[-\eta(a,-b),0]\times\bR$,
\item the Green's function of problem \eqref{gpabconst} changes sign in any other strip not a subset of the aforementioned.
\end{itemize}
\end{lem}
\begin{proof} The proof follows from the previous result together with the fact that
$$\eta(a,b)\le\frac{\pi}{2\omega}<\frac{\pi}{\omega}.$$
\end{proof}
\begin{rem} Realize that the rectangles defined in the previous Lemma are optimal in the sense that $G$ changes sign in a bigger rectangle. The same observation applies to the similar results we will prove for the other cases.  This fact implies that we cannot have maximum or anti-maximum principles on bigger intervals for the solution, something that is widely known and which the following results, together with Example \ref{exacrit} illustrate.
\end{rem}

Since $G(t,0)$ changes sign at $t = \eta(a,b)$. It is immediate to verify that by defining function $h_\epsilon(s)=1$ for all $s\in (-\epsilon,\epsilon)$ and $h(s)=0$ otherwise, we have a solution of problem \eqref{gpabconst} that cross the real value $c$ on the right of $\eta(a,b)$. So the estimates are optimal for this case.

However, one can study problems with particular non homogeneous part $h$ for which the solution has over $c$  for a bigger interval. This is showed in the following example.

\begin{exa}
Consider the problem $x'(t)-5 x(-t)+4 x(t)=\cos^2 3t$, $x(0)=0$.\par Clearly, we are in the case (C1).
For this problem,
\begin{align*}\bar{u}(t):= &\int_0^t\[\cos(3 (s - t)) +\frac{4}{3}\sin(3(s-t))\]\cos^2 3s\dif s-\frac{5}{3}\int_{-t}^0\sin(3(s+t))\dif s \\ = & \frac{1}{18}\(6\cos 3t+3\cos 6t +2\sin 3t+2\sin 6t-9\).
\end{align*}
$\bar{u}(0)=0$, so $\bar{u}$ is the solution of our problem.\par
Studying $\bar{u}$, we can arrive to the conclusion that $\bar{u}$ is non-negative in the interval $[0,\c]$, being zero at both ends of the interval and
$$\c=\frac{1}{3}\arccos\(\frac{1}{39}\[\sqrt[3]{47215-5265\sqrt{41}}+\sqrt[3]{5\(9443+1053\sqrt{41}\)}-35\]\)=0.201824\dots$$
Also, $\bar{u}(t)<0$ for $t=\c+\e$ with $\e\in\bR^+$ sufficiently small. Furthermore, the solution is periodic of period $2\pi/3$.
\end{exa}
\begin{figure}[h!t]
\center{\includegraphics[width=.5\textwidth]{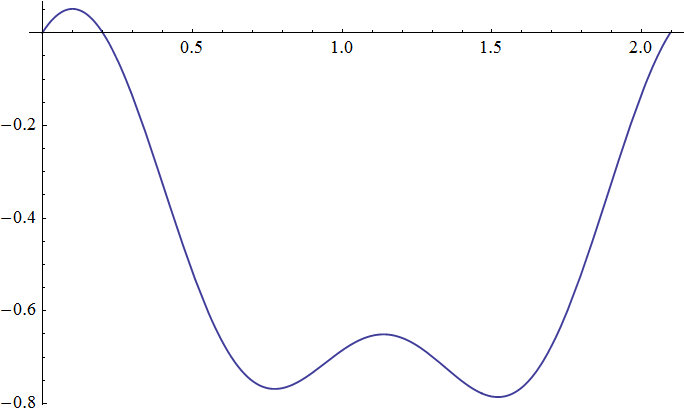}}\caption{Graph of the function $\bar{u}$ on the interval $[0,2\pi/3]$. Observe that $\bar{u}$ is positive on $(0,\c)$ and negative on $(\c,2\pi/3)$.}
\end{figure}\par
If we use Lemma \ref{lempma1}, we have that, a priori, $\bar{u}$ is non-positive on $[-4/15,0]$ which we know is true by the study we have done of $\bar{u}$, but this estimate is, as expected, far from the interval $[\c-1,0]$ in which $\bar{u}$ is non-positive. This does not contradict the optimality of the a priori estimate, as we have showed before, some other examples could be found for which the interval where the solution has constant is arbitrarily close to the  one given by the a priori estimate.
\subsection{The case (C2)}
We study here the case (C2). In this case, it is clear that
$$G(t,s)=\[\cosh(\omega (s - t)) + \frac{b}{\omega}\sinh(\omega(s-t))\]\chi_0^t(s)+\frac{a}{\omega}\sinh(\omega(s+t))\chi_{-t}^0(s),$$
which we can rewrite as
\begin{subequations}
\begin{empheq}[left={G(t,s)=\empheqlbrace}]{align}
   &  \cosh\omega (s - t) + \frac{b}{\omega}\sinh\omega(s-t), &  0\le s \le t, \label{G1b} \\
     & \label{G2b} -\cosh\omega (s - t) - \frac{b}{\omega}\sinh\omega(s-t), & t\le s \le 0, \\ 
     & \frac{a}{\omega}\sinh\omega(s+t), &  -t\le s \le 0, \label{G3b} \\
     & \label{G4b} -\frac{a}{\omega}\sinh\omega(s+t), &  0\le s \le -t, \\
       &    0, & \text{otherwise.}
 \end{empheq}
\end{subequations}
Studying the expression of $G$ we can obtain maximum and antimaximum principles.
With this information, we can state the following Lemma.

\begin{lem}\label{lempma2} Assume $a^2<b^2$ and define
$$\sigma(a,b):=\frac{1}{\sqrt{b^2-a^2}}\arctanh \frac{\sqrt{b^2-a^2}}{b}.$$
Then, 
\begin{itemize}
\item if $a>0$, the Green's function of problem \eqref{gpabconst} is positive on $\{(t,s), \; -t<s<0\}$ and $\{(t,s), \; 0<s<-t\}$,
\item  if $a<0$, the Green's function of problem \eqref{gpabconst} is negative on $\{(t,s), \; -t<s<0\}$ and $\{(t,s), \; 0<s<-t\}$,
\item if $b>0$, the Green's function of problem \eqref{gpabconst} is negative on $\{(t,s), \; t<s<0\}$,
\item if $b>0$, the Green's function of problem \eqref{gpabconst} is positive on $\{(t,s), \; 0<s<t\}$ if and only if $t \in (0,\sigma(a,b))$,
\item  if $b<0$, the Green's function of problem \eqref{gpabconst} is positive on $\{(t,s), \; 0<s<t\}$,
\item  if $b<0$, the Green's function of problem \eqref{gpabconst} is negative on $\{(t,s), \; t<s<0\}$ if and only if $t \in (\sigma(a,b),0)$.

\end{itemize}
\end{lem}
\begin{proof}For $0<a<b$, he argument of the $\sinh$ in \eqref{G4} is negative, so \eqref{G4b} is positive. The argument of the $\sinh$ in \eqref{G3} is positive, so \eqref{G3b} is positive. It is easy to check that \eqref{G1b} is positive as long as $t<\sigma(a,b)$.\par
On the other hand, \eqref{G2b} is always negative.\par
The rest of the proof continues similarly.
\end{proof}

As a corollary of the previous result we obtain the following one:

\begin{lem}\label{lempma} Assume  $a^2<b^2$. Then,
\begin{itemize}
\item if $0<a<b$, the Green's function of problem \eqref{gpabconst} is non-negative on $[0,\sigma(a,b)]\times\bR$,
\item if $b<-a<0$, the Green's function of problem \eqref{gpabconst} is non-negative on $[0,+\infty)\times\bR$,
\item if $b<a<0$, the Green's function of problem \eqref{gpabconst} is non-positive on $[\sigma(a,b),0]\times\bR$,
\item if $b>-a>0$, the Green's function of problem \eqref{gpabconst} is non-positive on $(-\infty,0]\times\bR$,
\item the Green's function of problem \eqref{gpabconst} changes sign in any other strip not a subset of the aforementioned.
\end{itemize}
\end{lem}

\begin{exa} Consider the problem
\begin{equation}\label{exaC2}x'(t)+\l x(-t)+2\l x(t)=e^t,\quad x(1)=c\end{equation}
with $\l> 0$. 

Clearly, we are in the case (C2).
$$\sigma(\l,2\l)=\frac{1}{\l\sqrt{3}}\ln[7+4\sqrt{3}]=\frac{1}{\l}\cdot 1.52069\dots$$
If $\l\ne 1/\sqrt{3}$, then
\begin{align*}\bar{u}(t): & =\int_0^t\[\cosh(\l\sqrt{3} (s - t)) + \frac{2}{\sqrt{3}}\sinh(\l\sqrt{3}(s-t))\]e^s\dif s+\frac{1}{\sqrt{3}}\int_{-t}^0\sinh(\omega(s+t))e^{s}\dif s\\ & = \frac{1}{3\l^2-1}\[(\l-1)(\sqrt{3}\sinh(\sqrt{3}\l t)-\cosh(\sqrt{3}\l t))+(2\l-1)e^{t}-\l e^{-t}\],
\end{align*}
%$$\bar{u}(1)=\frac{1}{3\l^2-1}\[(\l-1)(\sqrt{3}\sinh(\sqrt{3}\l )-\cosh(\sqrt{3}\l ))+(2\l-1)e-\l e^{-1}\],$$
$$\til u(t)=\cosh(\l\sqrt{3}t) -\sqrt{3}\sinh(\l\sqrt{3}t).$$

With these equalities, it is straightforward to construct the unique solution $w$ of problem \eqref{exaC2}. For instance, in the case $\l=c=1$,
$$\bar{u}(t)=\sinh(t),$$
and
$$w(t)=\sinh t+\frac{1-\sinh 1}{\cosh(\l\sqrt{3}) -\sqrt{3}\sinh(\l\sqrt{3})}\(\cosh(\l\sqrt{3}t) -\sqrt{3}\sinh(\l\sqrt{3}t)\).$$
Observe that for $\l=1,\ c=\sinh 1$, $w(t)=\sinh t$. Lemma \ref{lempma} guarantees the non-negativity of $w$ on $[0,1.52069\dots]$, but it is clear that the solution $w$ is positive on the whole positive real line.
\end{exa}
\subsection{The case (C3)}
We study here the case (C3) for $a=b$. In this case, it is clear that
$$G(t,s)=[1+a(s-t)]\chi_0^t(s)+a(s+t)\chi_{-t}^0(s),$$
which we can rewrite as
$$G(t,s)=\begin{cases}
  1+a(s-t), &  0\le s \le t, \\
     -1-a(s-t), & t\le s \le 0, \\ 
   a(s+t), &  -t\le s \le 0, \\
      -a(s+t), &  0\le s \le -t, \\
      0, & \text{otherwise.}
\end{cases}$$
Studying the expression of $G$ we can obtain maximum and antimaximum principles.
With this information, we can prove the following Lemma as we did with the analogous ones for cases (C1) and (C2).

\begin{lem}\label{lempma3} Assume $a=b$. Then, if $a>0$, the Green's function of problem \eqref{gpabconst} is
\begin{itemize}
\item positive on $\{(t,s), \; -t<s<0\}$ and $\{(t,s), \; 0<s<-t\}$,
\item negative on $\{(t,s), \; t<s<0\}$,
\item positive on $\{(t,s), \; 0<s<t\}$ if and only if $t \in (0,1/a)$,
\end{itemize}
and, if $a<0$, the Green's function of problem \eqref{gpabconst} is
\begin{itemize}
\item negative on $\{(t,s), \; -t<s<0\}$ and $\{(t,s), \; 0<s<-t\}$,
\item positive on $\{(t,s), \; 0<s<t\}$.
\item negative on $\{(t,s), \; t<s<0\}$ if and only if $t \in (1/a,0)$.
\end{itemize}
\end{lem}

As a corollary of the previous result we obtain the following one:

\begin{lem}\label{lempma3} Assume $a=b$. Then,
\begin{itemize}
\item if $0<a$, the Green's function of problem \eqref{gpabconst} is non-negative on $[0,1/a]\times\bR$,
\item if $a<0$, the Green's function of problem \eqref{gpabconst} is non-positive on $[1/a,0]\times\bR$,
\item the Green's function of problem \eqref{gpabconst} changes sign in any other strip not a subset of the aforementioned.
\end{itemize}

\end{lem}
For this particular case we have another way of computing the solution to the problem.

\begin{pro} Let $a=b$ and assume $2at_0\ne1$. Let $H(t):=\int_{t_0}^th(s)\dif s$ and $\cH(t):=\int_{t_0}^t H(s)\dif s$. Then problem \eqref{gpabconst} has a unique solution given by
$$u(t)=H(t)-2a\cH_o(t)+\frac{2a\,t-1}{2a\,t_0-1}c.$$
\end{pro}
\begin{proof}The equation is satisfied, since
$$u'(t)+a(u(t)+u(-t))=u'(t)+2a u_e(t)=h(t)-2\,a H_e(t)+\frac{2a\,c}{2a\,t_0-1}+2\, a H_e(t)-\frac{2a\,c}{2a\,t_0-1}=h(t).$$
The initial condition is also satisfied for, clearly, $u(t_0)=c$.
\end{proof}
\begin{exa}
Consider the problem $x'(t)+\l(x(t)-x(-t))=|t|^p,\ x(0)=1$ for $\l,p\in\bR$, $p>-1$. For $p\in(-1,0)$ we have a singularity at $0$. We can apply the theory in order to get the solution
$$u(t)=\frac{1}{p+1}t|t|^p+1-2\l t$$
where $\bar{u}(t)=\frac{1}{p+1}t|t|^p$ and $\til u(t)=1-2\l t$. $\bar{u}$ is positive in $(0,+\infty)$ and negative in $(-\infty,0)$ independently of $\l$, so the solution has better properties than the ones guaranteed by Lemma \ref{lempma3}.
\end{exa}
The next example shows that the estimate is sharp.
\begin{exa}\label{exacrit} Consider the problem
\begin{equation}\label{probexaC3} u_\e'(t)+u_\e(t)+u_\e(-t)=h_\e(t),\ t\in\bR;\quad u_\e(0)=0,\end{equation}
where $\e\in\bR$, $h_\e(t)=12x(\e-x)\chi_{[0,\e]}(x)$ and $\chi_{[0,\e]}$ is the characteristic function of the interval $[0,\e]$. Observe that $h$ is continuous. By means of the expression of the Green's function for problem \eqref{probexaC3}, we have that its unique solution is given by
$$u_\e(t)=\begin{cases} -2\e^3t-\e^4, & \text{if}\quad t<-\e, \\ -t^4-2\e t^3, & \text{if}\quad -\e<t<0, \\ t^4-(4+2\e)t^3+6\e t^2, & \text{if}\quad 0<t<\e, \\ -2\e^3 t+2\e^3+\e^4, & \text{if}\quad t>\e. \end{cases}$$
The a priory estimate on the solution tells us that $u_\e$ is non-negative at least in $[0,1]$. Studying the function $u_\e$, it is easy to check that $u_\e$ is zero at $0$ and $1+\e/2$, positive in $(-\infty,1+\e/2)\backslash\{0\}$ and negative in $(1+\e/2,+\infty)$.
\begin{figure}[h!t]
\center{\includegraphics[width=.5\textwidth]{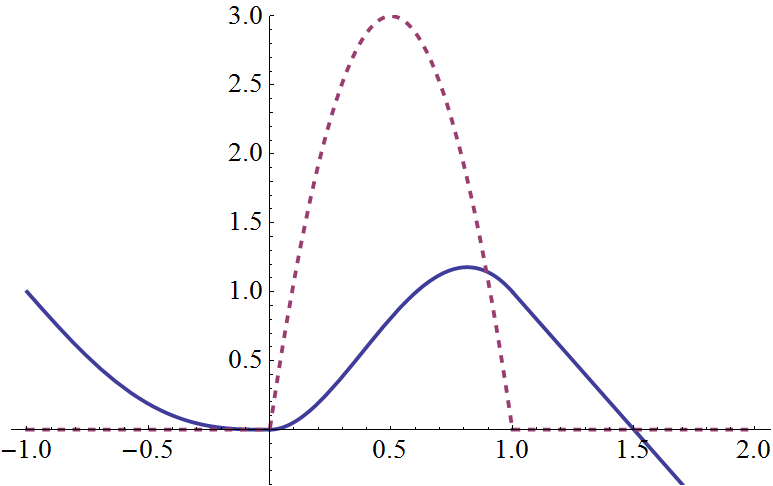}}\caption{Graph of the function $u_1$ and $h_1$ (dashed). Observe that $u$ becomes zero at $t=1+\e/2=3/2$.}
\end{figure}\par
\end{exa}

The case (C3.2) is very similar,
$$G(t,s)=\begin{cases}
  1+a(t-s), &  0\le s \le t, \\
     -1-a(t-s), & t\le s \le 0, \\ 
   a(s+t), &  -t\le s \le 0, \\
      -a(s+t), &  0\le s \le -t, \\
      0, & \text{otherwise.}
\end{cases}$$
\begin{lem}\label{lempma3} Assume $a=-b$. Then, if $a>0$, the Green's function of problem \eqref{gpabconst} is
\begin{itemize}
\item positive on $\{(t,s), \; -t<s<0\}$, $\{(t,s), \; 0<s<t\}$ and $\{(t,s), \; 0<s<-t\}$,
\item negative on $\{(t,s), \; t<s<0\}$ if and only if $t \in (-1/a,0)$,
\end{itemize}
and, if $a>0$, the Green's function of problem \eqref{gpabconst} is
\begin{itemize}
\item negative on $\{(t,s), \; -t<s<0\}$, $\{(t,s), \; t<s<0\}$ and $\{(t,s), \; 0<s<-t\}$,
\item positive on $\{(t,s), \; 0<s<t\}$ if and only if $t \in (0,-1/a)$.
\end{itemize}
\end{lem}

As a corollary of the previous result we obtain the following one:

\begin{lem}\label{lempma5} Assume $a=-b$. Then,
\begin{itemize}
\item if $a>0$,the Green's function of problem \eqref{gpabconst} is non-negative on $[0,+\infty)\times\bR$,
\item if $a<0$ the Green's function of problem \eqref{gpabconst} is non-positive on $(-\infty,0]\times\bR$,
\item the Green's function of problem \eqref{gpabconst} changes sign in any other strip not a subset of the aforementioned.
\end{itemize}
\end{lem}

Again, for this particular case we have another way of computing the solution to the problem.
\begin{pro} Let $a=-b$, $H(t):=\int_{0}^th(s)\dif s$ and $\cH(t):=\int_{0}^t H(s)\dif s$. Then problem \eqref{gpabconst} has a unique solution given by
$$u(t)=H(t)-H(t_0)-2a(\cH_e(t)-\cH_e(t_0))+c.$$
\end{pro}
\begin{proof}The equation is satisfied, since
$$u'(t)+a(u(t)-u(-t))=u'(t)+2\, a u_o(t)=h(t)-2\, a H_o(t)+2\, a H_o(t)=h(t).$$
The initial condition is also satisfied for, clearly, $u(t_0)=c$.
\end{proof}
\begin{exa}
Consider the problem $$x'(t)+\l(x(-t)-x(t))=\frac{\l t^2-2t+\l}{(1+t^2)^2},\ x(0)=\l$$ for $\l\in\bR$. We can apply the theory in order to get the solution
$$u(t)=\frac{1}{1+t^2}+\l(1+2\l t)\arctan t-\l^2\ln(1+t^2)+\l-1$$
where $\bar{u}(t)=\frac{1}{1+t^2}+\l(1+2\l t)\arctan t-\l^2\ln(1+t^2)-1$.

Observe that the real function
$$h(t):=\frac{\l t^2-2t+\l}{(1+t^2)^2}$$
is positive on $\bR$ if $\l > 1$ and negative  on $\bR$ for all $\l <-1$. Therefore, Lemma \ref{lempma5} guarantees that $\bar{u}$ will be positive on $(0, \infty)$ for $\l >1$ and in $(-\infty,0)$ when $\l < -1$.
\end{exa}

{\bf Acknowledgment.} The authors are thankful to the anonymous referees for the careful reading of the manuscript and suggestions.


\begin{thebibliography}{99}
%\bibitem{Abr}Abramowitz, M.; Stegun, I. A. (editors) \textit{Handook of Mathematical Functions with Formulas, Graphs, and Mathematical Tables}. Dover Publications (1972).


\bibitem{AzDo1} Azbelev, N.V., Domoshnitsky, A., A question concerning linear differential inequalities-I, Differentsial'nye uravnenija, 27, (1991), 257-263.

\bibitem{AzDo2} Azbelev, N.V., Domoshnitsky, A., A question concerning linear differential inequalities-II, Differentsial'nye uravnenija, 27, (1991), 641-647.


\bibitem{Aga} Agarwal,  R. P., Berezansky, L., Braverman, E., Domoshnitsky, A. \textit{Nonoscillation Theory of Functional Differential Equations with Applications}, Springer, New York, 2012.

\bibitem{Aft} Aftabizadeh, A. R.; Huang, Y. K.; Wiener; J. \textit{Bounded Solutions for Differential Equations with Reflection of the Argument}.  J. Math. Anal. Appl. 135 (1988), 31-37.
%\bibitem{agoreg} R. P. Agarwal and D. O'Regan,
%Upper and lower solutions for singular problems with nonlinear
%boundary data. \textit{NoDEA Nonlinear Differential Equations
%Appl.} \textbf{9} (2002), 419--440.
%
%\bibitem{amann} H. Amann, Fixed point equations and nonlinear
%eigenvalue problems in ordered Banach spaces, \textit{SIAM. Rev.},
%\textbf{18} (1976), 620--709.
\bibitem{And} Andrade, D.; Ma, T. F. \textit{Numerical solutions for a nonlocal equation with reflection of the argument.} Neural Parallel Sci. Comput. 10, (2002), 227-233.
%\bibitem{baxley} J. V. Baxley,
%A singular nonlinear boundary value problem: membrane response of
%a spherical cap. \textit{SIAM J. Appl. Math.} \textbf{48} (1988),
%497--505.
%
%\bibitem{Cab3} Cabada, Alberto. \textit{The Method of Lower and Upper Solutions for Second, Third, Fourth, and Higher Order Boundary Value Problems}.  J. Math. Anal. Appl. 185, No. 2 (1994), p.p. 302-320.
%\bibitem{Cab1} Cabada, Alberto; Cid, J. \'Angel. \textit{On Comparison Principles for the Periodic Hill's Equation}.  J. Lond. Math. Soc. (to appear) doi:10.1112/jlms/jds001
%\bibitem{Cab2} Cabada, Alberto; Cid, J. \'Angel; M\'aquez--Villamar\'in, Beatriz. \textit{Computation of Green's Functions for Boundary Value Problems with} Mathematica. (preprint)
\bibitem{Cab4}   Cabada, A.; Tojo, F. A. F. \textit{Comparison results for first order linear operators with reflection and periodic boundary value conditions}. Nonlinear Analysis: Theory, Methods and Applications. Vol. 78, (2013), 32--46.
\bibitem{Cab5} Cabada, A.; Infante, G.;  Tojo, F. A. F. \textit{Nontrivial Solutions of Perturbed Hammerstein Integral Equations with Reflections. Boundary Value Problems} 2013, 2013:86.
\bibitem{Dom1} Domoshnitsky, A. \textit{Maximum principles and nonoscillation intervals for first order Volterra functional differential equations}, Dynamics of Continuous, Discrete and Impulsive Systems. A. Mathematical Analysis, 15 (2008) 769--814.
\bibitem{Dom2} Domoshnitsky, A. \textit{Nonoscillation interval for $n$-th order functional differential equations}, Nonlinear Analysis, TMA 71(2009) e2449--e32456.
\bibitem{Dom3} Domoshnitsky, A., Maghakyan, A., Shklyar, R. \textit{Maximum principles and boundary value problems for first-order neutral functional differential equations}, J. Inequal. Appl. 2009, Art. ID 141959, 26 pp. 


%\bibitem{guolak} D. Guo and V. Lakshmikantham,
%\textit{Nonlinear Problems in Abstract Cones}, Academic Press, 1988.
%
%\bibitem{gm} P. Guidotti and S. Merino, Gradual loss of positivity and
%hidden invariant cones in a scalar heat equation,
%\textit{Differential Integral Equations}, \textbf{13} (2000),
%1551--1568.
%
%\bibitem{gnt}
%Ch. P. Gupta, S. K. Ntouyas and P. Ch. Tsamatos, Existence results
%for multi-point boundary value problems for second order ordinary
%differential equations, \textit{Bull. Greek Math. Soc.},
%\textbf{43} (2000), 105--123.
%
%\bibitem{gijwjiea} G. Infante and J. R. L. Webb, Three point
%boundary value problems with solutions that change sign,
%\textit{J. Integral Equations Appl.}, \textbf{15}, (2003), 37--57.
%%
%\bibitem{gijwjmaa} G. Infante and J. R. L. Webb, Nonzero solutions
%of Hammerstein Integral Equations with Discontinuous kernels,
%\textit{J. Math. Anal. Appl.}, \textbf{272}, (2002), 30--42.

%\bibitem{gijwems} G. Infante and J. R. L. Webb, Nonlinear nonlocal boundary value problems and
%perturbed Hammerstein integral equations, \textit{Proc. Edinb. Math.
%Soc.}, \textbf{49} (2006), 637--656.

%
%\bibitem{gijwnodea} G. Infante and J. R. L. Webb, Loss of positivity in
%a nonlinear scalar heat equation, to appear in \textit{NoDEA
%Nonlinear Differential Equations Appl.}.
%
%\bibitem{kttmna} G. L. Karakostas and P. Ch. Tsamatos,
%Existence of multiple positive solutions for a nonlocal boundary
%value problem, \textit{Topol. Methods Nonlinear Anal.} \textbf{19}
%(2002), 109--121.
%
%\bibitem{ktejde} G. L. Karakostas and P. Ch. Tsamatos,
%Multiple positive solutions of some Fredholm integral equations
%arisen from nonlocal boundary-value problems, \textit{Electron. J.
%Differential Equations} \textbf{2002}, No. 30, 17 pp.

%\bibitem{krzab} M. A. Krasnosel'ski\u\i{} and P. P. Zabre\u{\i}ko,
%{\it Geometrical methods of nonlinear analysis}, Springer-Verlag,
%Berlin, (1984).

%\bibitem{kljlms} K. Q. Lan,
%Multiple positive solutions of semilinear differential equations
%with singularities, \textit{J. London Math. Soc}, \textbf{63}
%(2001), 690--704.
%
%\bibitem{kljdeds} K. Q. Lan,
%Multiple positive solutions of Hammerstein integral equations with
%singularities, \textit{Differential Equations and Dynamical
%Systems}, \textbf{8} (2000), 175--195.
%
%\bibitem{kqljwjde} K. Q. Lan and J. R. L. Webb, Positive
%solutions of semilinear differential equations with singularities,
%\textit{J. Differential Equations}, \textbf{148} (1998), 407--421.
%
%\bibitem{macast} R. Ma and N. Castaneda, Existence of solutions of nonlinear
%$m$-point boundary value problems, \textit{J. Math. Anal. Appl.},
%\textbf{256} (2001), 556--567.
%
%\bibitem{martin} R. H. Martin, \textit{Nonlinear operators and differential equations in Banach spaces}, Wiley, New York, (1976).
%
%\bibitem{oregtmna} D. O'Regan, Upper and lower solutions for problems
%with singular sign changing nonlinearities and with nonlinear
%boundary data, \textit{Topol. Methods Nonlinear Anal.} \textbf{19}
%(2002),  375--390.
%
%\bibitem{oregwcna} D. O'Regan, Upper and lower solutions for singular
%problems arising in the theory of membrane response of a spherical
%cap. Proceedings of the Third World Congress of Nonlinear
%Analysts, Part 2 (Catania, 2000). \textit{Nonlinear Anal.}
%\textbf{47} (2001), 1163--1174.
%
%\bibitem{Tor} P. J. Torres, Existence of one-signed periodic solutions of some second-order differential equations via a Krasnoselskii fixed point theorem, \textit{J. Differential Equations}, \textbf{190} (2003) 643--662.

%\bibitem{jw-gi-jlms} J. R. L. Webb and G. Infante, Positive solutions of
%nonlocal boundary value problems: a unified approach, \textit{J. London Math. Soc.},
%(2) \textbf{74} (2006), 673--693.
%
%\bibitem{jwgi-nodea-08} J. R. L. Webb and G. Infante, Positive solutions
%of nonlocal boundary value problems involving integral conditions,
%\textit{NoDEA Nonlinear Differential Equations Appl.},  \textbf{15} (2008), 45--67.



%\bibitem{jwmz-na} J. R. L. Webb and M. Zima, Multiple positive solutions of resonant
%and non-resonant nonlocal boundary value problems, \textit{Nonlinear
%Anal.}, \textbf{71} (2009), 1369--1378.


%\bibitem{jwwcna00} J. R. L. Webb, Positive solutions of some three point
%boundary value problems via fixed point index theory. Proceedings
%of the Third World Congress of Nonlinear Analysts, Part 2
%(Catania, 2000). \textit{Nonlinear Anal.} \textbf{47} (2001),
%4319--4332.


%\bibitem{Aga} Agarwal, Ravi P.; O'Regan, Donal. \textit{An introduction to ordinary differential equations.} Universitext. Springer, New York, 2008.  
% \bibitem{And} Andrade, D.; Ma, T. F. \textit{Numerical solutions for a nonlocal equation with reflection of the argument.} Neural Parallel Sci. Comput. 10, (2002),227-233.

\bibitem{Gup}
Gupta, C. P. \textit{Existence and uniqueness theorems for boundary value problems involving reflection of the argument.} Nonlinear Anal. 11 (1987), 9, 1075-1083. 
\bibitem{Gup2}
Gupta, C. P. \textit{Two-point boundary value problems involving reflection of the argument.} Internat. J. Math. Math. Sci. 10 (1987), 2, 361-371.
%\bibitem{Kra} Krasnosel'ski\u{\i}, Mark A. \textit{Positive solutions of Operator Equations}. Noordhoff, Groningen, 1964.
\bibitem{Ma} Ma, T. F.; Miranda, E. S.; de Souza Cortes, M. B. \textit{A nonlinear differential equation involving reflection of the argument.} Arch. Math. (Brno) 40 (2004), 1, 63-68. 
\bibitem{Kul} Kuller, R. G. \textit{On the differential equation $f'=f \circ g$, where $g \circ g=I$}. Math. Mag. 42 (1969) 195-200.
\bibitem{Ore} O'Regan, D. \textit{Existence results for differential equations with reflection of the argument.} 
J. Austral. Math. Soc. Ser. A 57 (1994), 2, 237-260. 
\bibitem{Ore2} O'Regan, D.; Zima, Miroslawa. \textit{Leggett-Williams norm-type fixed point theorems for multivalued mappings.} Appl. Math. Comput. 187 (2007), 2, 1238-1249. 
\bibitem{Pia} Piao, D. \textit{Pseudo almost periodic solutions for differential equations involving reflection of the argument.} J. Korean Math. Soc. 41 (2004), 4, 747-754. 
\bibitem{Pia2} Piao, D. \textit{Periodic and almost periodic solutions for differential equations with reflection of the argument.} Nonlinear Anal. 57 (2004), 4, 633-637. 
%\bibitem{Sar} \v{S}arkovski\u{\i}, Alexander N. \textit{Functional-differential equations with a finite group of argument transformations.} (Russian) Akad. Nauk Ukrain. SSR, Inst. Mat., Kiev,  157, (1978), 118-142. 
\bibitem{Sha} Shah, S. M.; Wiener, J. \textit{Reducible functional-differential equations.} Internat. J. Math. Math. Sci. 8 (1985), 1-27.
\bibitem{Sil} Silberstein, L. \textit{Solution of the Equation $f'(x)=f(1/x)$}. Philos. Mag. 7:30 (1940), 185-186.
\bibitem{Wat1} Watkins, W. \textit{Modified Wiener Equations}. Int. J. Math. Math. Sci. 27:6 (2001), 347-356.
\bibitem{Wat2} Watkins, W. \textit{Asymptotic Properties of Differential Equations with Involutions}. Int. J. Pure Appl. Math. 44:4 (2008), 485-492.
%\bibitem{Wie3} Wiener, Joseph. \textit{Differential equations with involutions.} Differensial'nye Uravneniya, 5, (1969), 1131-1137. 
\bibitem{Wie1} Wiener, J.; Aftabizadeh, A. R. \textit{Boundary value problems for differential equations with reflection of the argument.} Internat. J. Math. Math. Sci. 8 (1985), 1, 151-163. 
\bibitem{Wie} Wiener J.; Watkins, W. \textit{A Glimpse into the Wonderland of Involutions}. Missouri J. Math. Sci. 14 (2002), 3, 175-185. 
\bibitem{Wie2} Wiener, J. \textit{Generalized solutions of functional-differential equations.} World Scientific Publishing Co., Inc., River Edge, NJ, 1993.
\end{thebibliography}
\end{document}